\documentclass{amsart}

\usepackage{hyperref}

\usepackage{amssymb}
\usepackage{amsmath}
\usepackage{amsthm}
\usepackage{enumerate}
\usepackage{graphicx}
\usepackage{tikz-cd}

\theoremstyle{plain}

\makeatletter
\newtheorem*{rep@theorem}{\rep@title}
\newcommand{\newreptheorem}[2]{%
\newenvironment{rep#1}[1]{%
 \def\rep@title{#2 \ref{##1}}%
 \begin{rep@theorem}}%
 {\end{rep@theorem}}}
\makeatother

\newtheorem{theorem}{Theorem}[section]
\newreptheorem{theorem}{Theorem}

\newtheorem{lemma}[theorem]{Lemma}
\newtheorem{corollary}[theorem]{Corollary}

\theoremstyle{definition}
\newtheorem{definition}[theorem]{Definition}
\newtheorem*{example}{Example}

\theoremstyle{remark}
\newtheorem*{remark}{Remark}

\newtheorem*{acknowledgments}{Acknowledgments}

\newcommand{\UnitDisk}{\mathbb{D}}
\newcommand{\ComplexPlane}{\mathbb{C}}
\newcommand{\Reals}{\mathbb{R}}
\newcommand{\NN}{\mathbb{N}}
\newcommand{\RiemannSphere}{\widehat{\mathbb{C}}}
\newcommand{\UnitCircle}{\mathbb{T}}

\renewcommand {\tilde} {\widetilde}
\renewcommand{\Re}{\mathrm{Re\,}}
\begin{document}

\title{Finitely connected domains, rational maps and Ahlfors functions}

\date{September 11, 2013}

\author[M. Fortier Bourque]{Maxime Fortier Bourque}
\thanks{First author supported by NSERC}
\address{Department of Mathematics, Graduate Center, City University of New York, 365 Fifth Avenue, New York (NY), USA, 10016.}
\email{maxforbou@gmail.com}

\author[M. Younsi]{Malik Younsi}
\thanks{Second author supported by the Vanier Canada Graduate Scholarships program.}
\address{D\'epartement de math\'ematiques et de statistique, Pavillon Alexandre--Vachon, $1045$ av. de la M\'edecine, Universit\'e Laval, Qu\'ebec (Qu\'ebec), Canada, G1V 0A6.}
\email{malik.younsi.1@ulaval.ca}


\begin{abstract}
Using Ahlfors functions, Grunsky maps and the Bell representation theorem, we show that a certain subset of the rational maps of degree $n$ forms a trivial bundle over the moduli space of non-degenerate $n$-connected domains with one marked tangent vector with fiber the $n$-fold symmetric pro\-duct of the circle. A consequence is that the set of rational Ahlfors functions of degree $n$ forms a closed embedded submanifold inside the space of rational maps of degree $n$. As an application, we show the existence of rational Ahlfors functions with non-positive residues, resolving a question left open in a previous paper by the authors.
\end{abstract}

\maketitle

\section{Introduction}

Let $n \geq 1$ be an integer. For $(a,b) \in \ComplexPlane^n \times \ComplexPlane^n$, let $R_{a,b}$ be the rational map defined by the formula
$$
z \mapsto \sum_{j=1}^n \frac{a_j}{z-b_j}.
$$
We denote by $\mathcal{R}(n)$ the set of $(a,b) \in \ComplexPlane^n \times \ComplexPlane^n$ such that $\sum_{j=1}^n b_j = 0$ and the sublevel set
$$
R_{a,b}^{-1}(\UnitDisk)=\{ z \in \RiemannSphere : |R_{a,b}(z)|<1 \}
$$
is connected and bounded by $n$ disjoint analytic Jordan curves. A variant of the set $\mathcal{R}(n)$ is considered in \cite{JEONG3} and \cite{JEONG2}, where it is called the \textit{coefficient body for Bell representations}. The set $\mathcal{R}(n)$ is clearly invariant under the action of the symmetric group $\Sigma_n$ and the quotient $\mathcal{R}(n)/\Sigma_n$ is in bijection with the set of rational maps $\{ R_{a,b} : (a,b) \in \mathcal{R}(n)\}$. The goal of this paper is to prove that $\mathcal{R}(n)$ forms a trivial $\UnitCircle^n$-bundle over a certain moduli space $\mathcal{M}(n)$.\\

We define $\mathcal{M}(n)$ to be the set of isomorphism classes of planar domains containing infinity and bounded by $n$ disjoint analytic Jordan curves, with an ordering of the boun\-dary curves. An \textit{isomorphism} between two such domains $(X,E_1,...,E_n)$ and $(Y,F_1,...,F_n)$  is by definition a biholomorphism $g: \overline X \to \overline Y$ such that $g(\infty)=\infty$, $g(z)/z \to 1$ as $z\to \infty$ and $g(E_j)=F_j$ for each $j$. Using Koebe's circle domain theorem (see e.g. \cite[Section 15.7]{CON}), we get that $\mathcal{M}(n)$ is in bijection with the set of $(c,r) \in \ComplexPlane^n \times (\Reals_{>0})^n$ such that $\sum_{j=1}^n c_j = 0$ and the $n$ disks $\{ z \in \ComplexPlane : |z-c_j| \leq r_j \}$ are disjoint. We use this bijection to put a topology on $\mathcal{M}(n)$. It is easy to see that $\mathcal{M}(n)$ is homotopy equivalent to the configuration space
$$
\mathcal{F}_n\ComplexPlane := \{ (c_1,...,c_n) \in \ComplexPlane^n : c_i \neq c_j\mbox{ for }i\neq j \}
$$
of $n$-tuples of distinct points in the plane. Therefore, the fundamental group of $\mathcal{M}(n)$ is the pure braid group on $n$ strands and the fundamental group of the quotient $\mathcal{M}(n)/\Sigma_n$ is the braid group on $n$ strands.\\

The relationship between $\mathcal{R}(n)$ and $\mathcal{M}(n)$ is expressed in the following theorem :
\begin{reptheorem}{BiebBellRep}
Let $[(X,E_1,...,E_n)] \in \mathcal{M}(n)$ and let $\alpha_j \in E_j$ for each $j$. Then there is a unique $(a,b) \in \mathcal{R}(n)$ and a unique isomorphism
$$
g: (X,E_1,...,E_n) \to (R_{a,b}^{-1}(\UnitDisk),F_1,...,F_n)
$$ such that the curve $F_j$ encloses $b_j$ and $R_{a,b}(g (\alpha_j))=1$ for each $j$.
\end{reptheorem}

This shows that $\mathcal{R}(n)$ is the same as the moduli space of planar domains containing infinity and bounded by $n$ disjoint analytic Jordan curves, with one marked point on each boundary curve.\\

Let $P : \mathcal{R}(n) \to \mathcal{M}(n)$ be the map which sends $(a,b)$ to the isomorphism class of the domain $R_{a,b}^{-1}(\UnitDisk)$ with boundary curves $F_1,...,F_n$ ordered in such a way that $F_j$ encloses the pole $b_j$ of $R_{a,b}$. Among all the preimages $(a,b)$ of a given $\sigma \in \mathcal{M}(n)$ by $P$, there is a unique one such that $R_{a,b}'(\infty)=\sum_{j=1}^n a_j$ is equal to the analytic capacity of $\RiemannSphere \setminus R_{a,b}^{-1}(\UnitDisk)$, so that $R_{a,b}$ is the Ahlfors function on $R_{a,b}^{-1}(\UnitDisk)$. This will be explained in sections 7, 8 and 9. The map $A : \mathcal{M}(n) \to \mathcal{R}(n)$ which sends $\sigma$ to this parameter $(a,b)$ is a right inverse for $P$.\\

Let
$$
\begin{array}{rcrcl}
\pi & : & \mathcal{M}(n)\times\UnitCircle^n &\to& \mathcal{M}(n)\\
    &   &   (\sigma,\beta) & \mapsto & \sigma
\end{array}
$$
and let
$$
\begin{array}{rcrcl}
\iota & : & \mathcal{M}(n) & \to & \mathcal{M}(n)\times \UnitCircle^n\\
      &   &  \sigma        & \mapsto & (\sigma,(1,...,1)).
\end{array}
$$ Our main result is :

\begin{reptheorem}{MainTheorem}
There is a homeomorphism
$$
H : \mathcal{R}(n) \to \mathcal{M}(n) \times \UnitCircle^n
$$
which commutes with the action of $\Sigma_n$ and is such that the diagrams
$$
\begin{tikzcd}
\mathcal{R}(n) \arrow{r}{H}  \arrow{dr}{P}
& \mathcal{M}(n) \times \UnitCircle^n \arrow{d}{\pi} \\
{}
& \mathcal{M}(n)
\end{tikzcd}
$$
and
$$
\begin{tikzcd}
\mathcal{R}(n) \arrow{r}{H}
& \mathcal{M}(n) \times \UnitCircle^n  \\
{}
& \mathcal{M}(n) \arrow{u}{\iota} \arrow{ul}{A}
\end{tikzcd}
$$
commute.
\end{reptheorem}

Theorem \ref{MainTheorem} supersedes \cite[Theorem 2.4]{JEONG2}, which shows that $\mathcal{R}(n)$ is homotopy equivalent to $\mathcal{F}_n \ComplexPlane  \times  \UnitCircle^n$.\\

A first consequence is that $P$ is continuous and open, and for every $\sigma \in \mathcal{M}(n)$ the inverse image $P^{-1}(\sigma)$ is homeomorphic to $\UnitCircle^n$. This answers \cite[Problem 4.2]{JEONG3} partially and supports a claim made at the end of \cite{JEONG2}.\\

Another consequence is that $A$ is a topological embedding with closed image. Therefore, the set $\mathcal{A}(n):=A(\mathcal{M}(n))$ of coefficients $(a,b)$ in $\mathcal{R}(n)$ such that $R_{a,b}$ is the Ahlfors function on $R_{a,b}^{-1}(\UnitDisk)$ forms a closed embedded subma\-nifold inside $\mathcal{R}(n)$. This gives a qualitative answer to \cite[Problem 1.5]{JEONG}.\\

Given $\sigma \in \mathcal{M}(n)$, the image $A(\sigma)$ is the parameter $(a,b) \in P^{-1}(\sigma)$ such that the sum $\sum_{j=1}^n a_j$ has largest real part. Intuitively, this maximizing parameter $(a,b)$ should have all summands $a_j$'s as nearly positive as possible. Let $\mathcal{R}^+(n)$ denote the subset of $\mathcal{R}(n)$ consisting of parameters $(a,b)$ such that all the $a_j$'s are real and positive, that is,
$$
\mathcal{R}^+(n) := \mathcal{R}(n) \cap ((\Reals_{>0})^n \times \ComplexPlane^n).
$$
We show in \cite{FBY} that
$$
\mathcal{A}(n)\cap (\Reals^n \times \Reals^n) = \mathcal{R}^+(n) \cap (\Reals^n \times \Reals^n)
$$
for all $n$ and
$$
\mathcal{A}(n)=\mathcal{R}^+(n)
$$
when $n$ is equal to $1$ or $2$, supporting the above intuition. \\

In the same paper, we give a numerical example showing that $\mathcal{R}^+(3)$ is not contained in $\mathcal{A}(3)$. As an application of Theorem \ref{MainTheorem}, we show by contradiction that the reverse inclusion does not hold either. The method is adapted to all $n \geq 3$.

\begin{reptheorem}{nonpositive}
For every $n\geq 3$, neither $\mathcal{R}^+(n) \subset \mathcal{A}(n)$ nor $\mathcal{A}(n) \subset \mathcal{R}^+(n)$.
\end{reptheorem}

The tools used to construct the bijection $H$ in Theorem \ref{MainTheorem} are Theorem \ref{BiebBellRep} and Ahlfors functions. The proof that $H$ is a homeomorphism requires several genera\-lizations of the Carath\'eodory kernel convergence theorem for finitely connected domains, some of which are new.\\

The paper is structured as follows. In section 2, we define the notion of Carath\'eodory kernel convergence for planar domains and state the generalized Carath\'eodory kernel convergence theorem. Then, in section 3, we review the definition of $\mathcal{M}(n)$ and prove a convergence theorem for Koebe representations onto circle domains. In section 4, we review the definition of $\mathcal{R}(n)$. Section 5 contains the proof of a convergence theorem for Bell representations. In section 6, we prove a compactness result for Grunsky maps. Next, in section 7, we prove a convergence theorem for Ahlfors functions. Sections 8 and 9 contain the proofs of Theorem \ref{MainTheorem} and Theorem \ref{nonpositive} respectively.

\section{Carath\'eodory kernel convergence}

In this paper, there will be many instances where given a domain $X \subset \RiemannSphere$ and some extra information, we will associate a uniquely defined analytic function $f$ on $X$. In all these cases, the function $f$ will depend continuously on the domain $X$ as well as the extra information. To make this precise, we need to formalize what it means for a sequence of domains to converge.

\begin{definition}
\label{defcara}
Let $\{X_k\}$ be a sequence of domains in $\RiemannSphere$ such that $\infty \in X_k$ for all $k$. We define the \textit{kernel} of $\{X_k\}$ (with respect to $\infty$), denoted by $\operatorname{ker}\{X_k\}$, to be the largest domain $X$ containing $\infty$ such that if $K$ is a compact subset of $X$, then there exists a $k_0$ such that $K \subseteq X_k$ for $k \geq k_0$, provided this set exists. Otherwise, we say that $\operatorname{ker}\{X_k\}$ does not exist.
\\

Moreover, we say that $\{X_k\}$ \textit{converges to $X$ (in the sense of Carath\'eodory)} if $X$ is the kernel of every subsequence of $\{X_k\}$. This is denoted by $X_k \rightarrow X$.
\end{definition}

\begin{definition}
Let $g$ be a function meromorphic in a neighborhood of $\infty$ in $\RiemannSphere$. We say that $g$ is \textit{tangent to the identity at infinity} if $g(\infty)=\infty$ and $\lim_{z \to \infty} g(z)/z=1$.
\end{definition}

Note that $g$ is tangent to the identity at infinity if and only if the function defined by $f(z):=1/g(1/z)$ satisfies $f(0)=0$ and $f'(0)=1$. The following lemma generalizes the theorem of Koebe which says that the family of schlicht functions on the unit disk is normal.

\begin{lemma}
\label{lemmeconv}
Let $\{X_k\}$ be as in Definition \ref{defcara}. For $k \geq 1$, let $g_k$ be univalent on $X_k$ and tangent to the identity at infinity. If $X=\operatorname{ker}\{X_k\}$ exists, then there is a subsequence $(g_{k_\ell})_{\ell=1}^{\infty}$ such that $X = \operatorname{ker}\{X_{k_\ell}\}$ and $(g_{k_\ell})_{\ell=1}^{\infty}$ converges locally uniformly to a function $g$ univalent on $X$.
\end{lemma}

\begin{proof}
See \cite[Lemma 15.4.6]{CON}.
\end{proof}

The next theorem was first proved by Carath\'eodory for simply connected domains.

\begin{theorem}[Generalized Carath\'eodory kernel convergence theorem]
\label{thmcara}
Let $\{X_k\}$ be as in Definition \ref{defcara}. For $k \geq 1$, let $g_k$ be univalent on $X_k$ and tangent to the identity at infinity. Assume that $X_k \rightarrow X$. Then $(g_k)_{k=1}^{\infty}$ converges locally uniformly on $X$ to a univalent function $g$ if and only if $\{g_k(X_k)\}$ converges to some domain $Y$. When this happens, $Y=g(X)$ and $g_k^{-1} \to g^{-1}$ locally uniformly on $Y$.
\end{theorem}

\begin{proof}
See \cite[Theorem 15.4.7]{CON}.
\end{proof}

\section{Koebe representations}

If $X\subset \RiemannSphere$ is a connected open set and its complement $\RiemannSphere \setminus X$ has $n$ connected components none of which is a point, then we say that $X$ is a \textit{non-degenerate $n$-connected (planar) domain}. Any non-degenerate $n$-connected domain can be mapped conformally onto a domain bounded by $n$ disjoint analytic Jordan curves. Such a conformal map can be obtained by applying the Riemann mapping theorem $n$ times, one for each boundary component. The function theory on such domains is often nicer. For example, any biholomorphism between domains bounded by analytic Jordan curves extends analytically to the boundary curves by the Schwarz reflection principle.\\

Accordingly, consider objects of the form $(X,E_1,...,E_n)$, where $X$ is a planar domain contai\-ning infinity and bounded by the disjoint analytic Jordan curves $E_1,...,E_n$. We define an \textit{isomorphism} between two such objects $(X,E_1,...,E_n)$ and $(Y,F_1,...,F_n)$ to be a biholomorphism $g: \overline X \to \overline Y$ such that
\begin{enumerate}[(i)]
\item $$g(\infty)=\infty;$$
\item $$\lim_{z\to \infty} g(z)/z = 1;$$
\item $$g(E_j)=F_j\mbox{ for each }j \in \{1,...,n\}.$$
\end{enumerate}

The set of isomorphism classes of such objects is denoted by $\mathcal{M}(n)$.\\

There are several families of special domains which represent all non-degenerate $n$-connected domains, see for example \cite[Chapter 15]{CON}. Among these, non-degenerate circle domains offer the advantage that each boundary component is an analytic Jordan curve.

\begin{definition}
A domain $Y \subset \RiemannSphere$ is called a \textit{circle domain} if each component of its complement is a spherical disk of non-negative radius.
\end{definition}

\begin{theorem}[Koebe]
Let $X$ be a non-degenerate $n$-connected domain. Then there exists a biholomorphism $g : X \to Y$ onto a non-degenerate circle domain. If $h : X \to Z$ is another biholomorphism onto a circle domain, then $h=M \circ g$ for some M\"obius transformation $M$.
\end{theorem}

Given $[(X,E_1,...,E_n)] $ in $\mathcal{M}(n)$, let $g: \overline X \to \overline Y$ be a biholomorphism onto a circle domain. By post-composing with an appropriate M\"obius transformation, we may further assume that $g$ is tangent to the identity at infinity. Then $Y$ is bounded by round circles $g(E_1),..., g(E_n)$ in the plane and $g$ is determined up to translation. Let $c_j$ and $r_j$ be the center and radius of the circle $g(E_j)$ in the euclidean metric on the plane. If we require that $\sum_{j=1}^n c_j = 0$, then $g$ is uniquely determined, and is called the \textit{normalized Koebe representation of $X$}. After this normalization, the map
$$
\begin{array}{rcrcl}
\kappa & : & \mathcal{M}(n) & \to & \ComplexPlane^{n-1} \times (\Reals_+)^n \\
  &   & [(X,E_1,...,E_n)] & \mapsto & ((c_1,...,c_{n-1}),(r_1,...,r_n))
\end{array}
$$
becomes well-defined and injective. Its image is the set of $(c,r) \in \ComplexPlane^{n-1} \times (\Reals_+)^n$ such that the $n$ disks $\{ z \in \ComplexPlane : |z-c_j|\leq r_j \}$ are disjoint, where $c_n:= - \sum_{j=1}^{n-1} c_j$. This is clearly an open set. We define the topology on $\mathcal{M}(n)$ to be the one induced by the map $\kappa$. Thus :

\begin{lemma} \label{dimmoduli}
The moduli space $\mathcal{M}(n)$ is a manifold of dimension $(3n-2)$ with a single chart.
\end{lemma}

\begin{remark}
By successively removing the conditions (iii), (ii) and (i) on isomorphisms, one obtains maps
$$
\mathcal{M}(n) \to \mathcal{M}(n)/\Sigma_n \to \mathcal{M}_{0,n,1} \to \mathcal{M}_{0,n,0},
$$
where $\Sigma_n$ is the symmetric group on $n$ symbols and $\mathcal{M}_{0,n,m}$ is the moduli space of Riemann surfaces of genus $0$ with $n$ disks removed and $m$ marked points. The spaces $\mathcal{M}_{0,n,1}$ and $\mathcal{M}_{0,n,0}$ are more natural to study from the point of view of Teichm\"uller theory, but the spaces $\mathcal{M}(n)$ and $\mathcal{M}(n)/\Sigma_n$ provide the appropriate framework for this paper. Note that one can think of $\mathcal{M}(n)/\Sigma_n$ as the moduli space of non-degenerate $n$-connected domains with one marked tangent vector.
\end{remark}

The following theorem provides a useful criterion to tell when a sequence converges in $\mathcal{M}(n)$ or $\mathcal{M}(n)/\Sigma_n$.

\begin{theorem}
\label{convergencekoebe}
Let $\{X_k\}$ be a sequence of $n$-connected domains containing $\infty$ converging in the sense of Carath\'eodory to an $n$-connected domain $X$. If $g_k : X_k \to Y_k $ are the normalized Koebe representations, then $(g_k)_{k=1}^{\infty}$ converges locally uniformly on $X$ to the normalized Koebe representation $g:X \to Y$. In particular, the sequence of circle domains $\{Y_k\}$ converges to $Y$ in the sense of Carath\'eodory.
\end{theorem}

\begin{proof}
We prove that every subsequence of $( g_k )_{k=1}^\infty$ has a subsequence which converges to $g$, which implies that $( g_k )_{k=1}^\infty$ converges to $g$.
\\

By Lemma \ref{lemmeconv}, every subsequence of $( g_k )_{k=1}^\infty$ has a subsequence converging locally uniformly to a univalent function on $X$.
\\

Let $h$ be a locally uniform limit of a subsequence $(g_k)_{k \in S}$. Then $h$ is tangent to the identity at infinity. By Theorem \ref{thmcara}, the corresponding subsequence of domains $\{Y_k\}_{k \in S}$ converges to $h(X)$ in the sense of Carath\'eodory.
\\

We claim that $h(X)$ is a non-degenerate circle domain bounded by circles whose centers sum to zero, so that $h=g$ by uniqueness of normalized Koebe representations. Indeed, first note that $h(X)$ is $n$-connected and non-degenerate since $h$ is univalent on $X$. Furthermore, the sequences of centers and radii of the circles bounding $\{Y_k\}_{k \in S}$ must be bounded, otherwise the kernel of $\{Y_k\}_{k \in S}$ would not contain a neighborhood of $\infty$. Therefore, passing to a subsequence if necessary, we can assume that the centers of the circles bounding $\{Y_k\}_{k \in S}$ converge to $c_1, c_2, \dots, c_n \in \mathbb{C}$ and that the corresponding radii converge to $r_1, r_2, \dots, r_n \in \mathbb{R}$. Since $Y_k \to h(X)$ as $k \rightarrow \infty$ in $S$, it is easy to see then that $h(X)$ is the domain bounded by the circles centered at $c_1, c_2, \dots, c_n$ of corresponding radii $r_1, r_2, \dots, r_n$. Since $h(X)$ is $n$-connected and non-degenerate, these circles must be disjoint and non-degenerate, so that $h(X)$ is indeed a non-degenerate circle domain. The fact that $\sum_{j=1}^{n} c_j =0$ follows from the fact that this is true for every domain in the sequence $\{Y_k\}$.\\

Finally, the fact that $Y_k \to Y$ follows directly from Theorem \ref{thmcara}.
\end{proof}

\begin{remark}
A similar argument shows that a sequence of $n$-connected circle domains $\{Y_k\}$ converges to a $n$-connected circle domain $Y$ if and only if the sequences of centers and radii bounding $\{Y_k\}$ converge to the centers and radii of the circles bounding $Y$.
\end{remark}

\section{Rational maps}
\label{rational}
First, we recall a simple lemma whose proof appears in \cite{FBY} :

\begin{lemma} \label{nconnected}
Let $R$ be a rational map of degree $n$ and $X:=R^{-1}(\UnitDisk)$. Then the following are equivalent :
\begin{enumerate}[(1)]
\item $\RiemannSphere \setminus X$ has $n$ connected components;
\item $X$ is $n$-connected;
\item $X$ is $n$-connected and non-degenerate;
\item $X$ is connected and bounded by $n$ disjoint analytic Jordan curves;
\item $R$ maps each component of $\RiemannSphere \setminus X$ homeomorphically onto $\RiemannSphere \setminus \UnitDisk$;
\item all the critical values of $R$ are in $\UnitDisk$.
\end{enumerate}
\end{lemma}

A rational map of degree $n$ which satisfies any (and hence all) of the conditions of Lemma \ref{nconnected} is called \textit{$n$-good}. Note that an $n$-good rational map must have only simple poles, since it cannot have $\infty$ as a critical value. An $n$-good map which vanishes at infinity and whose poles add up to zero is said to be \textit{normalized}. By definition, $\mathcal{R}(n)$ is the set of $(a,b)\in \ComplexPlane^n \times \ComplexPlane^n$ such that the rational map
$$
R_{a,b}(z) := \sum_{j=1}^n \frac{a_j}{z-b_j}
$$
is a normalized $n$-good map. Every normalized $n$-good map can be written as a sum of partial fractions and is thus equal to $R_{a,b}$ for some $(a,b) \in \mathcal{R}(n)$, unique up to permutation. Therefore, the set of normalized $n$-good maps can be identified with $\mathcal{R}(n) / \Sigma_n$.\\

By definition, if $(a,b) \in \mathcal{R}(n)$, then $\sum_{j=1}^n b_j = 0$. It follows that the map
$$
\begin{array}{rcrcl}
\theta & : & \mathcal{R}(n) & \to & \ComplexPlane^{n} \times \ComplexPlane^{n-1} \\
  &   & (a,b) & \mapsto & (a,(b_1,...,b_{n-1})).
\end{array}
$$
 is injective. Its image $\theta(\mathcal{R}(n))$ is  the set of $(a,\beta) \in  \ComplexPlane^n \times \ComplexPlane^{n-1}$ such that $R_{a,b}$  is $n$-good, where $b=(\beta,-\sum_{j=1}^{n-1} \beta_j)$. This image is open since the set of critical values of $R_{a,b}$ depends continuously on $(a,\beta)$ and $\UnitDisk$ is open. Thus :

\begin{lemma}\label{dimrational}
The space $\mathcal{R}(n)$ is a manifold of dimension $(4n-2)$ with a single chart.
\end{lemma}

For the quotient topology on the set of normalized $n$-good maps $\mathcal{R}(n)/\Sigma_n$, a sequence $(R_k)_{k=1}^\infty$ converges to $R$ if there exist $(a^k,b^k)$ and $(a,b)$ in $\mathcal{R}(n)$ such that $R_{a^k,b^k}=R_k$, $R_{a,b}=R$ and $(a^k,b^k) \to (a,b)$. It is easy to see that this implies that $(R_k)_{k=1}^\infty$ converges to $R$ uniformly on $\RiemannSphere$ with respect to the spherical metric. The converse is also true.

\begin{lemma}
\label{LemmaRational}
Let $(R_k)_{k=1}^{\infty}$ be a sequence of rational functions of degree $n$, each vanishing at infinity. Suppose that $(R_k)_{k=1}^{\infty}$ converges locally uniformly on some open set $U$ to a function $Q$ holomorphic on $U$. Then $Q$ is a rational function of degree at most $n$. If $Q$ has degree exactly $n$, then $R_k \to Q$ spherically uniformly on $\RiemannSphere$. Moreover, $Q$ vanishes at infinity and the poles and residues of $R_k$ converge to the poles and residues of $Q$.
\end{lemma}

\begin{proof}
Let us prove first that $Q$ is a rational function of degree at most $n$. Write $R_k=p_k/q_k$, where $p_k$ and $q_k$ are polynomials of degree at most $n$. Let $D$ be a closed disk contained in $U$ and fix a point $z_0$ in $D$ which is not a zero of any $q_k$. Multiplying $p_k$ and $q_k$ by the same constant $\lambda_k$ if necessary, we can assume that $\|q_k\|_{\infty,D} =1$ and that $q_k(z_0) > 0$ for all $k$. Then the $q_k$'s are in the closed unit ball of the finite-dimensional vector space of polynomials of degree at most $n$, which is compact. Consequently, there exists a subsequence $(q_{k_l})_{l=1}^{\infty}$ such that $q_{k_l} \to q$ uniformly on $D$, where $q$ is a polynomial of degree at most $n$. Moreover, we have that $\|q\|_{\infty,D}=1$ and $q(z_0) \geq 0$. Now, let $p:=Qq$. Since $p_k=R_k q_k$, $R_k \to Q$ and $q_{k_l} \to q$ uniformly on $D$, we get that $p_{k_l} \rightarrow p$ uniformly on $D$. By compactness, $p$ is a polynomial of degree at most $n$, as is every $p_k$. Hence we obtain that $Q$ is equal to the rational function $p/q$, which is of degree at most $n$.
\\

Assume now that the degree of $Q$ is exactly $n$. Let us prove first that $q_k \to q$ locally uniformly on $\mathbb{C}$ by showing that every subsequence of $(q_k)_{k=1}^{\infty}$ has a subsequence which converges locally uniformly to $q$. By the same argument as in the first part of the proof, every subsequence of $(q_k)_{k=1}^{\infty}$ has a subsequence which converges uniformly on $D$. Let $\tilde{q}$ be the uniform limit of a subsequence $(q_k)_{k \in S}$. Then $\tilde{q}$ is a polynomial of degree at most $n$, $\|\tilde{q}\|_{\infty,D}=1$ and $\tilde{q}(z_0) \geq 0$. As in the first part of the proof again, if $\tilde{p}:=Q \tilde{q}$, then the subsequence $(p_k)_{k \in S}$ converges to $\tilde{p}$ uniformly on $D$. Hence we have two polynomials of degree at most $n$, $\tilde{p}$ and $\tilde{q}$, such that $Q= \tilde{p}/ \tilde{q} = p/q$. The polynomials $\tilde{p}$ and $\tilde{q}$, as well as $p$ and $q$, cannot have a common factor since $Q$ has degree $n$. It follows that $\tilde{p}$ has the same zeros as $p$ and $\tilde{q}$ has the same zeros as $q$. There is thus a constant $\lambda$ such that $\tilde{p} = \lambda p$ and $\tilde{q} = \lambda q$. Since $\|\tilde{q}\|_{\infty,D}=\|q\|_{\infty,D}=1$, $\lambda$ is unimodular. As $Q$ is assumed to be holomorphic on $D$, we have $q(z_0)>0$ and $\tilde{q}(z_0)>0$, and hence $\lambda=1$. Therefore, the subsequence $(q_k)_{k \in S}$ converges to $q$ uniformly on $D$, and thus locally uniformly on all of $\mathbb{C}$, since all norms on a finite-dimensional vector space are equivalent.
\\

Now, since $p_k = R_k q_k$, $p=Q q$ and $R_k \to Q$, $q_k \to q$ on $D$, we get that $p_k \to p$ uniformly on $D$. Again by equivalence of norms, $p_k \to p$ locally uniformly on $\mathbb{C}$. One verifies easily that this implies that $R_k = p_k / q_k$ converges to $Q = p/q$ locally uniformly on $\ComplexPlane$ with respect to the spherical metric.
\\

Since $R_k$ vanishes at infinity, $q_k$ has degree $n$ and $p_k$ has degree at most $(n-1)$. Therefore, $p$ has degree at most $(n-1)$. The degree of $q$ is thus $n$, so that $Q=p/q$ vanishes at infinity.\\

Since $q_k \to q$ locally uniformly on $\mathbb{C}$, it follows from Rouch\'e's theorem that the zeros of $q_k$ converge to the zeros of $q$ as $k \to \infty$. Equivalently, the poles of $R_k$ converge to the poles of $Q$. Integrating $R_k$ on a small circle surrounding any pole of $Q$ and letting $k \to \infty$ shows that the residues of $R_k$ must converge to the residues of $Q$.\\

Let $B$ be a closed disk centered at $\infty$ on which $Q$ has no poles. When $k$ is large enough, $R_k$ has no poles on $B$. Since $R_k$ converges to $Q$ uniformly on $\partial B$, the maximum principle implies that convergence is uniform on $B$. Thus $R_k$ converges to $Q$ locally uniformly on all of $\RiemannSphere$. Since $\RiemannSphere$ is compact, convergence is uniform.

\end{proof}

For each $(a,b)\in\mathcal{R}(n)$, the domain $ R_{a,b}^{-1}(\UnitDisk)$ contains $\infty$ and is bounded by $n$ disjoint analytic Jordan curves by definition. Since there is exactly one pole of $R_{a,b}$ in each component of $\RiemannSphere \setminus R_{a,b}^{-1}(\UnitDisk)$, the ordering $(b_1,...,b_n)$ induces an ordering $(F_1,...,F_n)$ of the components of $\partial R_{a,b}^{-1}(\UnitDisk)$. This gives a map
$$
\begin{array}{rcrcl}
P &:& \mathcal{R}(n) &\to& \mathcal{M}(n)\\
  & &  (a,b) & \mapsto & [(R_{a,b}^{-1}(\UnitDisk),F_1,...,F_n)].
\end{array}
$$

As mentioned earlier, if $(a^k,b^k)$ converges to $(a,b)$, then $R_{a^k,b^k}$ converges spherically uniformly to $R_{a,b}$. An easy consequence is that the domains $R_{a^k,b^k}^{-1}(\UnitDisk)$ converge to $R_{a,b}^{-1}(\UnitDisk)$ in the sense of Carath\'eodory. It then follows from Theorem \ref{convergencekoebe} and the remark after it that $P(a^k,b^k)$ converges to $P(a,b)$ in $\mathcal{M}(n)$ and hence $P$ is continuous. \\

If $P$ is smooth and regular, as is claimed without proof in \cite{JEONG2}, then the inverse image $P^{-1}(\sigma)$ of every point $\sigma \in \mathcal{M}(n)$ is a smooth manifold of dimension $$
n= \dim \mathcal{R}(n) - \dim \mathcal{M}(n).$$ We do not know how to prove that $P$ is smooth, but we prove in section 8 that $P^{-1}(\sigma)$ is homeomorphic to the $n$-dimensional torus $\UnitCircle^n$ for every $\sigma \in \mathcal{M}(n)$.\\

In particular, $P$ is surjective. This means that for any $[(X,E_1,...,E_n)] \in \mathcal{M}(n)$, we can find $(a,b) \in \mathcal{R}(n)$ such that there is an isomorphism
$$
g : (X,E_1,...,E_n) \to (R_{a,b}^{-1}(\UnitDisk),F_1,...,F_n).
$$
If we have this, then the composition $R_{a,b} \circ g : X \to \UnitDisk$ is a proper holomorphic map of degree $n$ vanishing at infinity. Following \cite{BELL}, we call such a map a \textit{Grunsky map}. In section 5, we will see that every Grunsky map on $X$ arises in this way uniquely. In section 6, we will see that the set of Grunsky maps on $X$ is in bijection with the cartesian product of its boundary curves $E_1 \times \cdots \times E_n$.

\section{Bell representations}

Let $X$ be a non-degenerate $n$-connected planar domain containing infinity. Recall that a proper holomorphic map $f : X \to \UnitDisk$ of degree $n$ with $f(\infty)=0$ is called a Grunsky map.\\

The existence part of the following theorem was first proved in \cite{JEONG}. The uniqueness part was proved in \cite{FBY}.
\begin{theorem} \label{normalizedBellrep}
Let $X$ be a non-degenerate $n$-connected planar domain containing infinity and let $f$ be a Grunsky map on $X$. Then there exists a unique normalized $n$-good map $R\in \mathcal{R}(n)/\Sigma_n$ and a unique biholomorphism $g : X \to R^{-1}(\UnitDisk)$ tangent to the identity at infinity
such that $f=R\circ g$.
\end{theorem}

The pair $(R,g)$ in the above theorem is called the \textit{Bell representation of $X$ associated to $f$}. We now prove that Bell representations depend continuously on the input.

\begin{theorem}
\label{thmconv}
Let $X, X_k$ be non-degenerate $n$-connected planar domains each containing infinity. Let $f$ and $f_k$ be Grunsky maps on $X$ and $X_k$ respectively. Suppose that $X_k \to X$ and $f_k \to f$ locally uniformly on $X$. Write $f=R\circ g$ and $f_k=R_k \circ g_k$, where $R,R_k \in \mathcal{R}(n)/\Sigma_n$ and $g,g_k$ are biholomorphisms tangent to the identity at infinity. Then $g_k \to g$ locally uniformly on $X$ and $R_k \to R$ spherically uniformly on $\RiemannSphere$.
\end{theorem}
\begin{proof}
We prove that every subsequence $S\subset \NN$ admits a subsequence $S'\subset S$ along which the claimed convergence holds. Accordingly, let $S\subset \NN$ be any subsequence.\\

By Lemma \ref{lemmeconv}, $( g_k )_{k\in S}$ has a subsequence $( g_k )_{k\in S'}$ which converges locally uniformly to a univalent function $h$ on $X$. It follows that $h$ is tangent to the identity at infinity. Moreover, by Theorem \ref{thmcara}, the domains $g_k(X_k)$ converge to $h(X)$ and $g_k^{-1} \to h^{-1}$ locally uniformly on $h(X)$, as $k\to \infty$ in $S'$. Since $f_k \to f$ locally uniformly on $X$, we have that $R_k=f_k \circ g_k^{-1}$ converges to $Q:=f \circ h^{-1}$ locally uniformly on $h(X)$, as $k\to \infty$ in $S'$.
\\

By Lemma \ref{LemmaRational}, $Q$ is a rational function of degree at most $n$. Since $f$ has degree $n$ and $h$ is univalent, $Q$ has degree exactly $n$ and $h(X)$ is $n$-connected. Furthermore, since the restriction $Q : h(X) \to \UnitDisk$ is proper of degree $n$, $h(X)$ must be equal to the full preimage $Q^{-1}(\UnitDisk)$. By Lemma \ref{LemmaRational}, the poles of $R_k$ converge to the poles of $Q$. In particular, the poles of $Q$ sum to zero, as this is the case for every $R_k$. Therefore, $Q$ is a normalized $n$-good map.
\\

It follows from the uniqueness part of Theorem \ref{normalizedBellrep} that $Q=R$ and that $h=g$. Therefore, $g_k \to g$ locally uniformly on $X$ and $R_k \to R$ locally uniformly on $g(X)$, as $k\to \infty$ in $S'$. Since the subsequence $S\subset \NN$ was arbitrary, we have convergence as $k\to \infty$ in $\NN$. By Lemma \ref{LemmaRational}, $R_k$ converges to $R$ spherically uniformly on $\RiemannSphere$.
\end{proof}

In the next section, we will see how many Grunsky maps there are on any given non-degenerate $n$-connected planar domain containing infinity.

\section{Grunsky maps}

Let $X$ be a planar domain containing infinity and bounded by $n$ disjoint analytic Jordan curves $E_1,...,E_n$. If $f$ is a Grunsky map on $X$, then it extends analytically to the closure $\overline X$ by the Schwarz reflection principle. Moreover, the extended map $f : \overline X \to \overline \UnitDisk$ sends each boundary curve $E_j$ homeomorphically onto the unit circle. In particular, there exists a unique $\alpha_j \in E_j$ such that $f(\alpha_j)=1$ for each $j$. It turns out that each $n$-tuple $(\alpha_1,...,\alpha_n)\in E_1\times \cdots \times E_n$ arises in this way uniquely, and hence the set of Grunsky maps on $X$ can be parametrized by $E_1 \times \cdots \times E_n$.

\begin{theorem}
\label{normalizedBieberbach}
Let $[(X,E_1,...,E_n)] \in \mathcal{M}(n)$ and let $\alpha_j \in E_j$ for each $j$. There exists a unique Grunsky map on $X$ whose extension $f: \overline X \to \overline \UnitDisk$ satisfies $f(\alpha_j)=1$ for each $j$.
\end{theorem}
\begin{proof}
See \cite[Corollary 2.6]{FBY} and \cite[Theorem 2.2]{BELL}.
\end{proof}

We call the function $f$ in the above theorem the \textit{Grunsky map for $(\alpha_1,...,\alpha_n)$}. By combining the above result with Theorem \ref{normalizedBellrep}, we obtain the following.

\begin{theorem}\label{BiebBellRep}
Let $[(X,E_1,...,E_n)] \in \mathcal{M}(n)$ and let $\alpha_j \in E_j$ for each $j$. Then there is a unique $(a,b) \in \mathcal{R}(n)$ and a unique isomorphism
$$
g: (X,E_1,...,E_n) \to (R_{a,b}^{-1}(\UnitDisk),g(E_1),...,g(E_n))
$$ such that the curve $g(E_j)$ encloses $b_j$ for each $j$ and $R_{a,b} \circ g$ is the Grunsky map for $(\alpha_1,...,\alpha_n)$.
\end{theorem}

We will need the fact that Grunsky maps depend continuously on the domain $X$ as well as on the $n$-tuple $(\alpha_1,...,\alpha_n)$, at least for circle domains. The first ingredient for this is the following compactness result.

\begin{lemma} \label{GrunskyVariableDomain}
Let $X$ be an $n$-connected circle domain containing infinity. Suppose that $X_k \to X$, where each $X_k$ is an $n$-connected circle domain containing infinity. Let $f_k$ be a Grunsky map on $X_k$. Then there exists a subsequence $(f_{k_\ell})_{\ell=1}^{\infty}$, a Grunsky map $f$ on $X$, and a neighborhood $U$ of $\overline X$ such that $f_{k_\ell}$ and $f$ extend holomorphically to $U$ for all $\ell$ and $(f_{k_\ell})_{\ell=1}^\infty$ converges uniformly to $f$ on $U$.
\end{lemma}

\begin{proof}
Denote by $E_1,...,E_n$ the circles bounding $X$, let $J_j$ denote inversion in the circle $E_j$, and let
$$
Y:=\overline X \cup J_1(X) \cup \cdots \cup J_n(X).
$$
Similarly, denote by $E_1^k,...,E_n^k$ the circles bounding $X_k$ labeled in such a way that the center and radius of $E_j^k$ converge to the center and radius of $E_j$ as $k\to \infty$. Then let $J_j^k$ denote inversion in the circle $E_j^k$ and let
$$
Y_k:=\overline{X_k} \cup J_1^k(X_k) \cup \cdots \cup J_n^k(X_k).
$$
For each $j \in \{1,...,n\}$, we have that $J_j^k$ converges spherically uniformly to $J_j$ on $\RiemannSphere$. It follows that $Y_k \to Y$ in the sense of Carath\'eodory.  Moreover, each $f_k$ extends to a meromorphic function on $Y_k$, by Schwarz reflection. More precisely, if $J$ denotes inversion in the unit circle $\UnitCircle$, then for $z\in J_j^k(X_k)$ we define
$$f_k(z):=J(f_k(J_j^k(z))).$$

Let $F_k:= f_k^{-1}(\{-1,1,i\})$. For each $k$, $F_k$ has cardinality $3n$ and is contained in $\partial X_k$. Therefore, by passing to a subsequence if necessary, we can assume that $F_k$ converges to a finite set $F \subset \partial X$. By Montel's fundamental normality test, we can further extract a subsequence $(f_{k_\ell})_{\ell=1}^{\infty}$ converging locally uniformly to some meromorphic function $f$ on $Y \setminus F$ with respect to the spherical metric.\\

Since for each $k$ we have $f_k(X_k)=\UnitDisk$, we have that $f(X)\subset \overline \UnitDisk$ and hence $f$ is holomorphic on $X$. We also have $f_k(\partial X_k) = \UnitCircle$ for each $k$ and hence $f(\partial X \setminus F) \subset \UnitCircle$. Moreover, $f(\infty)=\lim_{\ell \to \infty} f_{k_\ell}(\infty)=0$, so that $f$ is not constant. By the maximum principle, we have $f(X) \subset \UnitDisk$.\\

Let $w\in \UnitDisk$ and let $\zeta$ be a zero of $f-w$ in $X$ of multiplicity $m$. Let $D \subset X$ be a closed disk centered at $z$ and such that $f-w$ does not vanish on $D\setminus \zeta$. Since $f_{k_\ell} \to f$ uniformly on $\partial D$, there exists an $L$ such that
$$|(f_{k_\ell}(z)-w) -(f(z) - w)| < |f(z)-w|$$
for all $z \in \partial D$ and all $\ell \geq L$. By Rouch\'e's theorem, $f_{k_\ell}-w$ then has $m$ zeros in $D$ counting multiplicity. Since $f_k-w$ has exactly $n$ zeros in $X_k$ for each $k$, this implies that $f-w$ has a total of at most $n$ zeros in $X$. Therefore, $f$ has degree at most $n$ on $X$. The same is true on $J_j(X)$ for each $j$, since $f$ must be equal to $J\circ f \circ J_j$ there. Note that the above argument also implies that the zeros of $f_{k_\ell}$ converge to the zeros of $f$ when $f$ has degree exactly $n$.\\

By Picard's big theorem, $f$ cannot have essential singularities in $F$, and thus extends to a meromorphic function on $Y$. By continuity, we have $f(\partial X) \subset \UnitCircle$. Since $f$ is not constant, the restriction $f : E_j \to \UnitCircle$ to each boundary component is open. The image $f(E_j)$ is therefore open in $\UnitCircle$ as well as compact and hence closed. By connectedness of the circle, we have $f(E_j)=\UnitCircle$ for each $j$. This means that $f$ has degree at least $n$ and thus exactly $n$ on $\overline X$. The map $f:X \to \UnitDisk$ is proper since it extends continuously to $\overline{X}$ with $f(\partial X) \subset \UnitCircle$. It is thus a Grunsky map.\\

Let $U$ be any neighborhood of $\overline X$ with closure in $Y$ on which $f$ is bounded. The identities $f_{k_\ell}(J_j^{k_\ell}(z))=J(f_{k_\ell}(z))$ and $f(J_j(z))=J(f(z))$ for $z \in X$ imply that $f_{k_\ell}$ has a zero at $z_0 \in X$ if and only if it has a pole at $J_j^{k_\ell}(z_0)$ for each $j$, and similarly for $f$. Since the zeros of $f_{k_\ell}$ converge to the zeros of $f$ and $J_j^{k_\ell}$ converges to $J_j$ locally uniformly on $\RiemannSphere$ for each $j$, it follows that the poles of $f_{k_\ell}$ must converge to the poles of $f$. We may thus assume that $f_{k_\ell}$ is holomorphic on $U$ for all $\ell$. By the maximum principle, the sequence $(f_{k_\ell})_{\ell=1}^{\infty}$ converges to $f$ uniformly on $\overline U$. Indeed, $\partial U$ is compact and contained in $Y\setminus F$, where $(f_{k_\ell})_{\ell=1}^{\infty}$ converges to $f$ locally uniformly.\\
\end{proof}

The following two corollaries are precisely what we need for the proof of Theorem \ref{MainTheorem}.

\begin{corollary} \label{GrunskyBoundaryPoints1}
Let $(X_k,E_1^k,...,E_n^k)$ and $(X,E_1,...,E_n)$ be non-degenerate $n$-connec\-ted circle domains containing infinity. Suppose that the center and radius of the circle $E_j^k$ converge to the center and radius of the circle $E_j$, as $k\to \infty$ for each $j$. Let $\alpha^k \in \prod_{j=1}^n E_j^k$ and $\alpha \in \prod_{j=1}^n E_j$ be such that $\alpha^k \to \alpha$. Let $f_k$ and $f$ be the Grunsky maps on $X_k$ and $X$ for $\alpha^k$ and $\alpha$ respectively. Then $(f_k)_{k=1}^{\infty}$ converges to $f$ locally uniformly on $X$.
\end{corollary}
\begin{proof}
Let $(f_k)_{k\in S}$ be any subsequence of $(f_k)_{k=1}^{\infty}$. By Lemma \ref{GrunskyVariableDomain}, we can extract a subsequence $S' \subset S$ such that there is a Grunsky map $\varphi$ on $X$ and a neighborhood $U$ of $\overline X$ such that $\varphi$ and $f_k$ extend analytically to $U$ for all $k\in S'$ and $(f_k)_{k \in S'}$ converges to $\varphi$ uniformly on $U$ as $k \to \infty$ in $S'$. For each $j \in \{1,...,n \}$, we thus have
$$
\varphi(\alpha_j)=\lim_{\substack{k\to\infty \\ k \in S'}} f_k(\alpha_j^k) = 1,
$$
and hence $\varphi=f$ by Theorem \ref{normalizedBieberbach}. This proves that every subsequence of $(f_k)_{k=1}^{\infty}$ has a subsequence which converges to $f$ locally uniformly on $X$. Therefore $(f_k)_{k=1}^{\infty}$ converges to $f$ locally uniformly on $X$.
\end{proof}

\begin{corollary} \label{GrunskyBoundaryPoints2}
Let $(X_k,E_1^k,...,E_n^k)$ and $(X,E_1,...,E_n)$ be $n$-connected circle domains containing infinity. Suppose that the center and radius of the circle $E_j^k$ converge to the center and radius of the circle $E_j$ for each $j$ as $k\to \infty$. Let $f_k$ and $f$ be Grunsky maps on $X_k$ and $X$ respectively. Let $\beta^k,\beta \in \UnitCircle^n$ be such that $\beta^k \to \beta$ and let $\alpha_j^k \in E_j^k$ and $\alpha_j \in E_j$ be such that $f_k(\alpha_j^k)=\beta_j^k$ and $f(\alpha_j)=\beta_j$ for each $j$. If $(f_k)_{k=1}^{\infty}$ converges to $f$ locally uniformly on $X$, then $\alpha^k \to \alpha$.
\end{corollary}

\begin{proof}
Fix $j \in \{1,...,n\}$, and let $(\alpha_j^k)_{k \in S}$ be any subsequence of $(\alpha_j^k)_{k\in \NN}$. Since the circle $E_j^k$ converges to $E_j$, we can extract a subsequence $S' \subset S$ such that $\alpha_j^k$ converges to some $z_j \in E_j$ as $k\to \infty$ in $S'$. Moreover, by Lemma \ref{GrunskyVariableDomain}, there exists a further subsequence $S'' \subset S'$, a Grunsky map $\varphi$ on $X$ and a neighborhood $U$ of $\overline X$ to which $\varphi$ and $f_k$ extend for all $k\in S''$ and such that $(f_k)_{k \in S''}$ converges to $\varphi$ uniformly on $U$ as $k\to \infty$ in $S''$. Since $f_k$ converges locally uniformly to $f$ on $X$, we have $\varphi=f$. Moreover, by uniform convergence, we have that
$$
f(z_j)=\lim_{\substack{k\to\infty \\ k\in S''}} f_k(\alpha_j^k)= \lim_{\substack{k\to\infty \\ k\in S''}} \beta_j^k = \beta_j.
$$
Since the restriction $f : E_j \to \UnitCircle$ is injective, $z_j = \alpha_j$. Therefore, every subsequence of $(\alpha_j^k)_{k\in \NN}$ has a subsequence converging to $\alpha_j$ and hence $(\alpha_j^k)_{k=1}^{\infty}$ converges to $\alpha_j$.\\
\end{proof}

\section{Ahlfors functions}

If $g$ is a function holomorphic in a neighborhood of $\infty$ in $\RiemannSphere$, then we define
$$
v_\infty(g):=g'(\infty):=\lim_{z\to \infty} z(g(z)-g(\infty)).
$$

Let $X$ be a planar domain containing $\infty$. The \textit{analytic capacity} of $\RiemannSphere \setminus X$ is defined as
$$
\gamma(\RiemannSphere \setminus X) := \sup \{ |g'(\infty)| : g \in \mathcal{O}(X,\overline \UnitDisk) \},
$$
where $\mathcal{O}(X,Y)$ denotes the set of holomorphic maps from $X$ to $Y$. The analytic capacity  $\gamma(\RiemannSphere \setminus X)$ is also known as the \textit{Carath\'eodory length} of the tangent vector $v_\infty$ in $X$.\\

If $\gamma(\RiemannSphere \setminus X)>0$, then there is a unique $f \in \mathcal{O}(X,\UnitDisk)$ such that
$$
f'(\infty) = \gamma(\RiemannSphere \setminus X),
$$
called the \textit{Ahlfors function} on $X$. It is easy to see that the Ahlfors function satisfies $f(\infty)=0$. Furthermore, if $h$ is univalent on $X$ and satisfies $h(\infty)=\infty$ and $\lim_{z\to \infty} h(z)/z = a$, then the Ahlfors function on $h(X)$ is $(|a|/a) f \circ h^{-1}$ and $$\gamma(\RiemannSphere \setminus h(X))=|a|\gamma(\RiemannSphere \setminus X).$$ We refer to this property as the \textit{transformation law} for analytic capacity. Finally, analytic capacity is \textit{outer regular}, in the sense that if $K_1 \supset K_2 \supset K_3 \dots$ is a decreasing sequence of compact sets, then $\gamma(\cap_j K_j) = \lim_{j \rightarrow \infty} \gamma(K_j)$.
\\

The following is due to Ahlfors \cite{AHL} :
\begin{theorem}[Ahlfors] \label{Ahlfors}
If $X$ is a non-degenerate $n$-connected domain containing $\infty$, then the Ahlfors function on $X$ is a Grunsky map. In particular, if $X$ is bounded by $n$ disjoint analytic Jordan curves, then the Ahlfors function on $X$ extends analytically to a neighborhood of $\overline X$.
\end{theorem}

We will need continuous dependence of Ahlfors functions on their domain at least when the limiting domain is bounded by analytic Jordan curves. This is false in general even if each domain considered is a non-degenerate circle domain.

\begin{example}
Let $X:= \RiemannSphere \setminus \overline{\mathbb{D}}$ and let $\{x_k\}_{k=1}^{\infty}$ be a sequence dense in $\UnitDisk$. Define $X_k$ to be the complement in $\RiemannSphere$ of disjoint closed disks centered at $x_1,x_2, \dots, x_k$ and contained in $\mathbb{D}$ of radius sufficiently small so that the analytic capacity of $\RiemannSphere \setminus X_k$ is less than $1/2$. This is always possible by outer-regularity of analytic capacity and by the fact that the analytic capacity of a finite set is zero, by Riemann's removable singularity theorem and Liouville's theorem. Then it is easy to see that $X_k \to X$ in the sense of Carath\'eodory, but the corresponding Ahlfors functions $f_k$ do not converge locally uniformly to the Ahlfors function on $X$, for otherwise we would have $\gamma(\RiemannSphere \setminus X_k) \to \gamma(\RiemannSphere \setminus X)=1$.
\end{example}

We thus need a stronger notion of convergence for domains.

\begin{definition}
Let $X$ and $X_k$ be domains in $\RiemannSphere$. We say that $X_k$ \textit{converges strongly to} $X$, and write $X_k \rightrightarrows X$, if for every compact set $K \subset X$ and every open set $U \supset \overline X$, we have $K \subset X_k \subset U$ for all but finitely many $k$.
\end{definition}

Note that if $\infty \in X$ and $X_k$ converges strongly to $X$, then $X_k$ converges to $X$ in the sense of Carath\'eodory.

\begin{lemma} \label{AhlforsContinuous}
Let $X$ be a planar domain containing $\infty$ and bounded by finitely many analytic Jordan curves. Suppose that $X_k \rightrightarrows X$, where $X_k$ are arbitrary domains containing $\infty$. Let $f_k$ and $f$ be the Ahlfors functions on $X_k$ and $X$ respectively. Then $(f_k)_{k=1}^{\infty}$ converges locally uniformly to $f$ on $X$.
\end{lemma}
\begin{proof}
By Ahlfors' theorem, $f$ extends holomorphically to some neighborhood $V$ of $\overline X$. Then take a neighborhood $U\supset \overline X$ with $\overline U \subset V$, so that $f$ is bounded on $U$. \\

We show that every subsequence of $( f_k )_{k=1}^\infty$ has a subsequence which converges to $f$, which implies that $(f_k)_{k=1}^{\infty}$ converges to $f$.\\

By Montel's theorem, the uniformly bounded sequence $( f_k )_{k=1}^\infty$ forms a normal family. Therefore, every subsequence of $( f_k )_{k=1}^\infty$ has a subsequence which converges locally uniformly to a holomorphic function on $X$.\\

Let $g$ be the locally uniform limit of a subsequence. Then $g \in \mathcal{O}(X,\overline\UnitDisk)$, so $g'(\infty) \leq f'(\infty)$.\\

By hypothesis, $X_k \rightrightarrows X$, so if $k$ is large enough we have $X_k \subset U$ and $f$ is defined and holomorphic on $X_k$. Let $M_k := \sup \{|f(z)| : z \in X_k\}$. Then $M_k^{-1}f \in \mathcal{O}(X_k,\overline \UnitDisk)$, so that $M_k^{-1}f'(\infty) \leq f_k'(\infty)$. Of course, $M_k \to 1$ as $k \to \infty$, since $ X_k\rightrightarrows X$ and $f$ is continuous on $U$. Therefore,
$$
f'(\infty) \leq \liminf f_k'(\infty) \leq g'(\infty)
$$
and thus $g'(\infty)=f'(\infty)$. By uniqueness of the Ahlfors function, we have $g=f$.
\end{proof}

If we require that each domain $X_k$ in the sequence has the same connectivity as $X$, then we can replace strong convergence by Carath\'eodory convergence.

\begin{theorem}\label{nconnectedAhlfors}
Let $X$ be a non-degenerate $n$-connected domain containing infinity. Suppose that $X_k \to X$, where each $X_k$ is a non-degenerate $n$-connected domain containing infinity. Let $f_k$ and $f$ be the Ahlfors functions on $X_k$ and $X$ respectively. Then $(f_k)_{k=1}^{\infty}$ converges locally uniformly to $f$ on $X$.
\end{theorem}

\begin{proof}
Let $h_k : X_k \to Y_k$ and $h : X \to Y$ be normalized Koebe representations. Then $(h_k)_{k=1}^{\infty}$ converges locally uniformly to $h$ on $X$ and $Y_k \to Y$, by Theorem \ref{convergencekoebe}. Since $Y_k$ and $Y$ are circle domains of connectivity $n$, we have in fact $Y_k \rightrightarrows Y$. Let $\varphi_k$ and $\varphi$ be the Ahlfors functions on $Y_k$ and $Y$ respectively. By Lemma \ref{AhlforsContinuous}, $(\varphi_k)_{k=1}^{\infty}$ converges to $\varphi$ locally uniformly on $Y$. By the transformation law, we have $f_k = \varphi_k \circ h_k$ and $f=\varphi \circ h$, so that $(f_k)_{k=1}^{\infty}$ converges to $f$ locally uniformly on $X$.
\end{proof}

\begin{remark}
By the transformation law, the Carath\'eodory length of $v_\infty$ is a well-defined function on the moduli space $\mathcal{M}(n)/\Sigma_n$. The above theorem implies that this function is continuous. A similar but more general result is true for the Kobayashi--Poincar\'e length (see \cite{HEJ}).
\end{remark}

\section{Rational Ahlfors functions}

Let $[X]\in \mathcal{M}(n)/\Sigma_n$. The Ahlfors function $f$ on $X$ is a Grunsky map by Theorem \ref{Ahlfors}. By Theorem \ref{normalizedBellrep}, there exists a unique $R \in \mathcal{R}(n)/\Sigma_n$ and a unique biholomorphism
$$
g : X \to  R^{-1}(\UnitDisk)
$$
tangent to the identity at infinity such that $f=R \circ g$. By the transformation law, $R$ is the Ahlfors function on $g(X)=R^{-1}(\UnitDisk)$.

\begin{definition}
If $Q \in \mathcal{R}(n)/\Sigma_n$ is such that $Q$ is the Ahlfors function on $Q^{-1}(\UnitDisk)$, we say that $Q$ is a \textit{normalized rational Ahlfors function}.
\end{definition}

Let $Q \in \mathcal{R}(n)/\Sigma_n$ be a normalized rational Ahlfors function such that there is a biholomorphism
$$
h : X \to Q^{-1}(\UnitDisk)
$$
tangent to the identity at infinity. Then $Q \circ h$ is the Ahlfors function on $X$ by the transformation law, and thus $Q=R$ by the uniqueness part of Theorem \ref{normalizedBellrep}. Accordingly, we say that $R$ is the \textit{normalized rational Ahlfors function associated to X}.\\

If the curves bounding $X$ are labelled $E_1,...,E_n$, then we may define $A(X,E_1,...,E_n)$ as the unique $(a,b) \in \mathcal{R}(n)$ such that $R_{a,b}=R$ and $b_j$ is contained in $g(E_j)$ for each $j$. If $$h: (X,E_1,...,E_n) \to (Y,F_1,...,F_n)$$ is an isomorphism, then the Ahlfors function on $Y$ is given by $f \circ h^{-1}$ by the transformation law, and the latter factors uniquely as $R_{a,b} \circ (g \circ h^{-1})$. This shows that the map
$$
\begin{array}{rcrcl}
A &:& \mathcal{M}(n) &\to& \mathcal{R}(n)\\
  & & [(X,E_1,...,E_n)] & \mapsto & (a,b)
\end{array}
$$
is well-defined. Moreover, the image $\mathcal{A}(n):=A(\mathcal{M}(n))$ is the set of $(a,b) \in \mathcal{R}(n)$ such that $R_{a,b}$ is a normalized rational Ahlfors function.
\\

The map $A$ is a right inverse for the map $P$ defined in section 4, since
$$
g : (X,E_1,...,E_n) \to (R_{a,b}^{-1}(\UnitDisk),g(E_1),...,g(E_n))
$$
is an isomorphism and $P(a,b)$ is by definition the isomorphism class of the latter.\\

We can use the map $A$ to construct a homeomorphism between $\mathcal{M}(n) \times \UnitCircle^n$ and $\mathcal{R}(n)$. In turn, this will shed light on the topological properties of $A$.

\begin{theorem}\label{MainTheorem}
There is a homeomorphism
$$
H : \mathcal{R}(n) \to \mathcal{M}(n) \times \UnitCircle^n
$$
which commutes with the action of $\Sigma_n$ and is
such that the diagrams
\begin{equation} \label{diagram1}
\begin{tikzcd}
\mathcal{R}(n) \arrow{r}{H}  \arrow{dr}{P}
& \mathcal{M}(n) \times \UnitCircle^n \arrow{d}{\pi} \\
{}
& \mathcal{M}(n)
\end{tikzcd}
\end{equation}
and
\begin{equation} \label{diagram2}
\begin{tikzcd}
\mathcal{R}(n) \arrow{r}{H}
& \mathcal{M}(n) \times \UnitCircle^n  \\
{}
& \mathcal{M}(n) \arrow{u}{\iota} \arrow{ul}{A}
\end{tikzcd}
\end{equation}
commute, where $\pi :  \mathcal{M}(n) \times \UnitCircle^n \to \mathcal{M}(n)$ is the projection onto the first factor and $\iota: \mathcal{M}(n) \to \mathcal{M}(n) \times \UnitCircle^n$ is the inclusion of $\mathcal{M}(n)$ as $\mathcal{M}(n) \times \{(1,...,1)\}$.
\end{theorem}
\begin{proof}
Let $(a,b)\in \mathcal{R}(n)$. We define $H=(H_1,H_2)$ as follows. We take
$$
H_1(a,b):=P(a,b)=[(R_{a,b}^{-1}(\UnitDisk), F_1,...,F_n)] \in \mathcal{M}(n),
$$
where $F_j$ is taken to be the boundary of the component of $R_{a,b}^{-1}(\RiemannSphere \setminus \UnitDisk)$ containing $b_j$. Let $\alpha_1,...,\alpha_n$ be the points in $F_1,...,F_n$ respectively with $R_{a,b}(\alpha_j)=1$. Also let $f: \overline{R_{a,b}^{-1}(\UnitDisk)} \to \overline{\UnitDisk}$ be the Ahlfors function. We set $$H_2(a,b):=(f(\alpha_1),...,f(\alpha_n)) \in \UnitCircle^n.$$ By construction, $H_2(a,b)=(1,...,1)$ if and only if $R_{a,b}$ is the Ahlfors function on $R_{a,b}^{-1}(\UnitDisk)$, i.e. if and only if $(a,b)=A(P(a,b))$.\\

We now construct an inverse $G$ to $H$. Given $[(X,E_1,...,E_n)] \in \mathcal{M}(n)$, let $f$ be the Ahlfors function on $X$. For $(\beta_1,...,\beta_n) \in \UnitCircle^n$, let $\alpha_j$ be the unique point in $E_j$ such that $f(\alpha_j)=\beta_j$. By Theorem \ref{BiebBellRep}, there is a unique $(a,b) \in \mathcal{R}(n)$ and a unique isomorphism $g : (X,E_1,...,E_n) \to (R_{a,b}^{-1}(\UnitDisk),F_1,...,F_n)$ such that $R_{a,b}(g(\alpha_j))=1$ for each $j$. We then set $
G([(X,E_1,...,E_n)],(\beta_1,...,\beta_n)):=(a,b).$\\

It is straightforward to verify that $G$ is the inverse of $H$, that both maps commute with the action of $\Sigma_n$, and that the diagrams commute.\\

By Lemma \ref{dimmoduli} and Lemma \ref{dimrational}, $\mathcal{R}(n)$ and $\mathcal{M}(n) \times \UnitCircle^n$ are both manifolds of dimension $(4n-2)$. By Brouwer's invariance of domain, $H$ is continuous if and only if $G$ is.\\

Let us prove that $G$ is continuous. Let $(X_k,E_1^k,...,E_n^k)$ be a sequence of normalized circle domains converging to $(X,E_1,...,E_n)$ and let $\beta^k \in \UnitCircle^n$ converge to $\beta$. Let $f_k$ and $f$ be the Ahlfors functions on $X_k$ and $X$ respectively. Let $\alpha_j^k$ be the point in $E_j^k$ such that $f_k(\alpha_j^k)=\beta_j^k$ and let $\alpha_j$ be the point in $E_j$ such that $f(\alpha_j)=\beta_j$.\\

By Theorem \ref{nconnectedAhlfors}, we have that $f_k$ converges to $f$ locally uniformly on $X$. By Corollary \ref{GrunskyBoundaryPoints2}, this implies that $\alpha^k \to \alpha$.\\

Let $g_k$ be the Grunsky map for $\alpha^k$ and $g$ the Grunsky map for $\alpha$. Then $g_k$ converges to $g$ locally uniformly on $X$ by Corollary \ref{GrunskyBoundaryPoints1}.\\

Let $(a^k,b^k)$ and $(a,b)$ be the parameters for the factorizations $R_{a^k,b^k} \circ h_k = g_k$ and $R_{a,b} \circ h = g$. Then $R_{a^k,b^k}$ converges to $R_{a,b}$ spherically uniformly on $\RiemannSphere$ by Theorem \ref{thmconv}. By Lemma \ref{LemmaRational}, we have $(a^k,b^k) \to (a,b)$.

\end{proof}

It follows that the maps $P$ and $A$ have the same topological properties as $\pi$ and $\iota$.

\begin{corollary}
The map $P$ is continuous and open, and for every $\sigma \in \mathcal{M}(n)$ the set $P^{-1}(\sigma)$ is homeomorphic to $\UnitCircle^n$.
\end{corollary}

\begin{corollary} \label{ContinuousSection}
The map $A : \mathcal{M}(n) \to \mathcal{R}(n)$ is a topological embedding with closed image.
\end{corollary}

\begin{remark}
All the arrows in diagrams \ref{diagram1} and \ref{diagram2} commute with the action of the symmetric group $\Sigma_n$. Therefore, they descend to the quotients and the resulting diagrams
$$
\begin{tikzcd}
\mathcal{R}(n)/\Sigma_n \arrow{r}{\tilde H}  \arrow{dr}{\tilde P}
& (\mathcal{M}(n)/\Sigma_n) \times (\UnitCircle^n/\Sigma_n) \arrow{d}{\tilde \pi} \\
{}
& \mathcal{M}(n)/\Sigma_n
\end{tikzcd}
$$
and
$$
\begin{tikzcd}
\mathcal{R}(n)/\Sigma_n \arrow{r}{\tilde H}
& (\mathcal{M}(n)/\Sigma_n) \times (\UnitCircle^n/\Sigma_n)  \\
{}
& \mathcal{M}(n)/\Sigma_n \arrow{u}{\tilde \iota} \arrow{ul}{\tilde A}
\end{tikzcd}
$$
commute. The action of $\Sigma_n$ on $\mathcal{M}(n)$, $\mathcal{R}(n)$ and $\mathcal{A}(n)$ is properly disconti\-nuous and without fixed points. Therefore, the quotients $\mathcal{M}(n)/\Sigma_n$, $\mathcal{R}(n)/\Sigma_n$ and $\mathcal{A}(n)/\Sigma_n$ are manifolds without boundary. However, when $n>1$, the action of $ \Sigma_n$ on $\UnitCircle^n$ has fixed points and the quotient $\UnitCircle^n / \Sigma_n$ is a manifold with non-empty boundary (see \cite{MORTON}). For example, $\UnitCircle^2 / \Sigma_2$ is the M\"obius band with boundary.

\end{remark}

\section{The positivity criterion}

The manifold $\mathcal{A}(n):=A(\mathcal{M}(n))$  in  $\ComplexPlane^n \times \ComplexPlane^n$ represents all Ahlfors functions on all non-degenerate $n$-connected domains. It is therefore of interest to determine this manifold explicitly.\\

Let $\mathcal{R}^+(n)$ denote the subset of $\mathcal{R}(n)$ consisting of parameters $(a,b)$ such that all the $a_j$'s are real and positive, that is, $\mathcal{R}^+(n) := \mathcal{R}(n) \cap ((\Reals_{>0})^n \times \ComplexPlane^n)$.\\

Here is a heuristic argument explaining why one should expect $\mathcal{A}(n)$ and $\mathcal{R}^+(n)$ to have anything to do with each other. For every $(a,b) \in \mathcal{R}(n)$, we have
$$
R_{a,b}'(\infty) = \lim_{z\to \infty} z \sum_{j=1}^n \frac{a_j}{z-b_j}= \sum_{j=1}^n a_j.
$$
Given $\sigma = [(X,E_1,...,E_n)]$ in $\mathcal{M}(n)$ and $(a,b) \in P^{-1}(\sigma)$ with biholomorphism
$$
g : X \to R_{a,b}^{-1}(\UnitDisk)
$$
tangent to the identity at infinity,
we have that $R_{a,b} \circ g \in \mathcal{O}(X,\UnitDisk)$ and hence
$$
\Re \sum_{j=1}^n a_j = \Re R_{a,b}'(\infty) \leq |R_{a,b}'(\infty)|= |(R_{a,b}\circ g)'(\infty)| \leq \gamma(\RiemannSphere \setminus X).
$$
Moreover, if the equality $\Re \sum_{j=1}^n a_j = \gamma(\RiemannSphere \setminus X)$ occurs, then $R_{a,b}\circ g$ is the Ahlfors function on $X$, so that $(a,b)=A(\sigma)$. In other words, $A(\sigma)$ is the unique parameter $(a,b) \in P^{-1}(\sigma)$ maximizing the quantity $\Re \sum_{j=1}^n a_j$. Intuitively, this parameter should have all $a_j$'s nearly positive in order to maximize the real part of the sum $\sum_{j=1}^n a_j$. Of course, this depends on the shape of the $n$-dimensional torus $P^{-1}(\sigma)$ sitting inside $\ComplexPlane^n \times \ComplexPlane^n$.\\

In fact, $\mathcal{A}(n)$ is equal to $\mathcal{R}^+(n)$ for $n=1,2$, as shown in \cite{FBY}. However, this fails in higher connectivity :

\begin{lemma} \label{generalizedexample}
For every $n\geq 3$, $\mathcal{R}^+(n) \setminus \mathcal{A}(n)$ is not empty.
\end{lemma}

\begin{proof}
In \cite{FBY}, we give a numerical example of a $3$-good rational map with positive residues which is not Ahlfors. The specific example is
$$
R(z) : = \frac{0.4}{z} + \frac{0.4}{z-(1+i)} + \frac{0.4}{z-6}.
$$
\\

Choose distinct points $b_4,...,b_n$ in $R^{-1}(\UnitDisk)$. For $\varepsilon>0$, define
$$
Q_\varepsilon(z) := R(z) + \sum_{j=4}^n \frac{\varepsilon}{z-b_j}.
$$
Our claim is that when $\varepsilon$ is small enough, $Q_\varepsilon$ is $n$-good, but is not Ahlfors.\\

We have that $Q_\varepsilon \to R$ locally uniformly on $\RiemannSphere \setminus \{ b_4 ,..., b_n \}$ as $\varepsilon \to 0$. This implies that $Q_\varepsilon^{-1}(\UnitDisk)$ converges strongly to $R^{-1}(\UnitDisk) \setminus \{ b_4 ,..., b_n\}$. In particular, when $\varepsilon$ is small enough, $Q_\varepsilon^{-1}(\RiemannSphere \setminus \UnitDisk)$ has at least $n$ connected components and thus exactly $n$ since $Q_\varepsilon$ has degree $n$. Therefore, $Q_\varepsilon$ is $n$-good. \\

Let $f$ be the Ahlfors function on $R^{-1}(\UnitDisk) \setminus \{ b_4 ,..., b_n\}$. The singularities $b_4 ,..., b_n$ are removable so that $f$ is the Ahlfors function on $R^{-1}(\UnitDisk)$. Since $R$ is not the Ahlfors function on $R^{-1}(\UnitDisk)$, we have $f'(\infty) > R'(\infty)$. Now let $f_\varepsilon$ be the Ahlfors function on $Q_\varepsilon^{-1}(\UnitDisk)$. As in the proof of Lemma \ref{AhlforsContinuous}, we have that $f_\varepsilon \to f$ locally uniformly on $R^{-1}(\UnitDisk) \setminus \{ b_4 ,..., b_n\}$ as $\varepsilon \to 0$. In particular, we have $f_\varepsilon'(\infty) \to f'(\infty)$.\\

On the other hand, $Q_\varepsilon'(\infty) = R'(\infty) + (n-3) \varepsilon \to R'(\infty)$ as $\varepsilon \to 0$. If $\varepsilon$ is small enough, we thus have
$$
f_\varepsilon'(\infty) > Q_\varepsilon'(\infty)
$$
so that $Q_\varepsilon$ is not the Ahlfors function on $Q_\varepsilon^{-1}(\UnitDisk)$. Precomposing $Q_\varepsilon$ with the appropriate translation yields an element in $(\mathcal{R}^+(n)/\Sigma_n) \setminus (\mathcal{A}(n)/\Sigma_n)$.

\end{proof}

We will prove that $\mathcal{A}(n)$ is not contained in $\mathcal{R}^+(n)$ either for $n\geq 3$. Beforehand, we need topological information on $\mathcal{R}^+(n)$.

\begin{lemma}
Let $K \subset \mathcal F_n \ComplexPlane$ be a compact set. There exists an $\varepsilon>0$ such that if $b \in K$ and $a=(a_1,...,a_n)$ satisfies $0<|a_j|\leq\varepsilon$ for each $j\in\{1,...,n\}$, then $R_{a,b}$ is $n$-good.
\end{lemma}
\begin{proof}
For $(a,b) \in \ComplexPlane^n \times \ComplexPlane^n$, we define the critical radius $\rho(a,b)$ to be the modulus of the largest critical value of $R_{a,b}$. This is a continuous function. Note that $R_{\lambda a, b} = \lambda R_{a,b}$, so that $\rho(\lambda a , b) = |\lambda| \rho(a,b)$. Let
$$
B:= \{ a \in \ComplexPlane^n : \|a\|_\infty \leq 1 \}.
$$
Then $B$ is compact so that $\rho$ attains a maximum $M$ on $B \times K$. If $0< |a_j| \leq 1/2M$ for all $j$, then $2M a \in B$, so that for every $b\in K$ the critical radius of $(2M a, b)$ is at most $M$. The critical radius of $(a, b)$ is thus at most $1/2$, and hence $R_{a,b}$ is $n$-good since it has degree $n$ and all its critical values are in the unit disk.
\end{proof}

\begin{lemma} \label{PositiveConnected}
$\mathcal{R}^+(n)$ is a closed connected submanifold of $\mathcal{R}(n)$ of dimension $(3n-2)$.
\end{lemma}
\begin{proof}
The fact that $\mathcal{R}^+(n)$ is a closed submanifold of $\mathcal{R}(n)$ of dimension $(3n-2)$ is elementary.\\

We now prove that $\mathcal{R}^+(n)$ is path connected. Let $(a^0,b^0),(a^1,b^1) \in \mathcal{R}^+(n)$. To construct a path we first shrink the vector of residues $a^0$ while keeping the poles fixed. Once the residues are all small enough, we move the poles from $b^0$ to $b^1$ while adjusting the residues individually in such a way that we can then expand them by a common factor to end up with $a^1$.\\

It is well-known (and easy to prove by induction) that $\mathcal{F}_n\ComplexPlane$ is path-connected. Let $p : [0,1] \to \mathcal{F}_n\ComplexPlane$ be a path from $b^0$ to $b^1$. We can modify $p$ so that the sum of its coordinates is zero for all $t$. In other words, define $q(t)_j:= p(t)_j - \frac{1}{n}\sum_{i=1}^n p(t)_i$. Let $K:=q([0,1])$ denote the trace of the path, and let $0<\varepsilon \leq \min(\|a^0\|_\infty,\|a^1\|_\infty)$ be as in the previous lemma.\\

For each $j\in{1,...,n}$, define
$$
b_j^t := \left\{
\begin{array}{ccl}
b_j^0  & & t \in [0,1/3) \\
q(3t-1)_j & & t \in [1/3,2/3] \\
b_j^1 & & t\in (2/3,1]
\end{array}\right.
$$
and
$$
a_j^t := \left\{
\begin{array}{ccl}
\mu_t a_j^0  & & t \in [0,1/3) \\
(2-3t)\mu_{1/3} a_j^0 + (3t-1) \nu_{2/3} a_j^1 & & t \in [1/3,2/3] \\
\nu_t a_j^1 & & t\in (2/3,1]
\end{array}\right.,
$$
where
$$
\mu_t := 1-3(1- \varepsilon / \|a^0\|_\infty)t
$$
and
$$
\nu_t:= 3\varepsilon/\|a^1\|_\infty-2+3(1- \varepsilon / \|a^1\|_\infty)t
$$
are positive constants varying affinely with $t$ such that $\mu_0=\nu_1=1$, $\mu_{1/3} a_j^0 \leq \varepsilon$ and $\nu_{2/3} a_j^1 \leq \varepsilon$ for all $j$.\\

The path $(a^t,b^t)$ is clearly continuous and is such that the points $b_1^t,...,b_n^t$ are distinct and add up to zero and $a_1^t,...,a_n^t >0$ for all $t$. The map $R_{a^t,b^t}$ is $n$-good for $t\in [0,1/3)$ since
$$
R_{a^t,b^t} = \mu_t R_{a^0,b^0},
$$
and $\mu_t\in (0,1]$. It is $n$-good for $t\in (2/3,1]$ since
$$
R_{a^t,b^t} = \nu_t R_{a^1,b^1},
$$
and $\nu_t \in (0,1]$. Finally, $R_{a^t,b^t}$ is $n$-good for $t \in [1/3,2/3]$ because $b^t \in K$ and $0<a_j^t \leq \varepsilon$ for each $j$.\\

Therefore, $(a^t,b^t)$ is a path from $(a^0,b^0)$ to $(a^1,b^1)$ inside $\mathcal{R}^+(n)$.
\end{proof}

We can now prove :

\begin{lemma} \label{reverseinclusion}
For every $n\geq 3$, $\mathcal{A}(n) \setminus  \mathcal{R}^+(n)$ is not empty.
\end{lemma}
\begin{proof}
Suppose for a contradiction that $\mathcal{A}(n)$ is contained in $\mathcal{R}^+(n)$. The map $A : \mathcal{M}(n) \to \mathcal{R}^+(n)$ is then an embedding with closed image by Corollary \ref{ContinuousSection} and Lemma \ref{PositiveConnected}. Since $\mathcal{M}(n)$ and $\mathcal{R}^+(n)$ are both manifolds of dimension $(3n-2)$, the map $A$ is open by Brouwer's invariance of domain. Since $\mathcal{R}^+(n)$ is connected, $A$ must be surjective, which contradicts Lemma \ref{generalizedexample}.
\end{proof}

\begin{remark}
The above proof is non-constructive. Indeed, we don't know any explicit example of a rational Ahlfors function having non-positive residues.
\end{remark}

Lemma \ref{generalizedexample} and Lemma \ref{reverseinclusion} together yield :

\begin{theorem} \label{nonpositive}
For every $n\geq 3$, neither $\mathcal{R}^+(n) \subset \mathcal{A}(n)$  nor $\mathcal{A}(n) \subset \mathcal{R}^+(n)$.
\end{theorem}

\acknowledgments{The authors thank Jeremy Kahn and Thomas Ransford for helpful discussions.}

\bibliographystyle{amsplain}

\end{document}